\documentclass[reqno,11pt]{amsart}
\usepackage{amsthm}
\newcommand{\bigH}{\mathcal{H}}
\newcommand{\bigF}{\mathcal{F}}

\renewcommand{\theequation}

\def\H{\mathbb H}
\def\R{\mathbb R}

\def\C{\mathbb C}

\def\(s){\mathscr S(\R\times\R)}

\newtheorem{thm}{Theorem}[section]

 \newtheorem{prop}[thm]{Proposition}
 \theoremstyle{definition}
 \newtheorem{defn}[thm]{Definition}
 \theoremstyle{remark}
 \newtheorem{rem}[thm]{Remark}
 
 \numberwithin{equation}{section}

\newcommand{\rr}{\mathbb{R}}

\def\C{\mathbb C}

\begin{document}

\title[Fock space in the slice hyperholomorphic setting]
{The Fock space in the slice hyperholomorphic setting}

\author[D. Alpay]{Daniel Alpay}
\address{(DA) Department of Mathematics\\
Ben-Gurion University of the Negev\\
Beer-Sheva 84105 Israel} \email{dany@math.bgu.ac.il}
\author[F. Colombo]{Fabrizio Colombo}
\address{(FC) Politecnico di
Milano\\Dipartimento di Matematica\\ Via E. Bonardi, 9\\20133
Milano, Italy}
\email{fabrizio.colombo@polimi.it}
\author[I. Sabadini]{Irene Sabadini}
\address{(IS) Politecnico di
Milano\\Dipartimento di Matematica\\ Via E. Bonardi, 9\\20133
Milano, Italy}
\email{irene.sabadini@polimi.it}
\author[G. Salomon]{Guy Salomon}
\address{(GS) Department of Mathematics\\
Ben-Gurion University of the Negev\\
Beer-Sheva 84105 Israel} \email{guysal@math.bgu.ac.il}

\thanks{D. Alpay thanks the Earl Katz family for endowing the chair
which supported his research, and the Binational Science
Foundation Grant number 2010117. F. Colombo and I. Sabadini
acknowledge the Center for Advanced Studies of the Mathematical
Department of the Ben-Gurion University of the Negev for the
support and the kind hospitality during the period in which part
of this paper has been written.}
\subjclass{MSC:  30G35, 30H20}
\keywords{Fock space, slice hyperholomorphic functions, quaternions, Clifford algebras.}

\begin{abstract}
In this paper we introduce and study some basic properties of the Fock space
(also known as Segal-Bargmann space) in the slice hyperholomorphic setting.
We discuss both
 the case of slice regular functions over quaternions
and also the case of slice monogenic functions with values in a
Clifford algebra. In the specific setting of quaternions, we also
introduce the full Fock space. This paper can be seen as the
beginning of the study of infinite dimensional analysis in the
quaternionic setting.
\end{abstract}

\maketitle

\section{Introduction}
Fock spaces are a very important tool in quantum mechanics, and
also in its quaternionic formulation; see the book of Adler
\cite{adler} and the paper \cite{MR1352898}. Roughly speaking, they can be
seen as the completion of the direct sum of the symmetric or
anti-symmetric, or full tensor powers of a Hilbert space which, from the
point of view of Physics, represents a single particle. There is
an alternative description of the Fock spaces in the holomorphic
setting which, in this framework, are also known as
Segal-Bargmann spaces.\smallskip

In this note we work  first in the setting of slice
hyperholomorphic functions, namely either we work with slice
regular functions (these are functions defined on subsets of the
quaternions with values in the quaternions) or with slice
monogenic functions (these functions are defined on the Euclidean
space $\mathbb{R}^{n+1}$ and have values in the Clifford algebra
$\mathbb{R}_n$), see the book \cite{bookfunctional}.\smallskip

Slice hyperholomorphic functions have been introduced quite
recently but they have already several applications, for example
in Schur analysis and to define some functional calculi. The
application to Schur analysis  started with the paper \cite{acs1}
and it is rapidly growing, see for example \cite{abcsNP, acls,
acsxx, acs2, acs3}.\smallskip

The applications to the functional calculus ranges from the
so-called S-functional calculus, which works for $n$-tuples non
necessarily commuting operators, to a quaternionic version of the
classical Riesz-Dunford functional calculus, see \cite{ds}. The
literature on slice hyperholomorphic functions and the related
functional calculi is wide, and we refer the reader to the book
\cite{bookfunctional} and the references therein.\smallskip

We note that  Fock spaces have been treated in the  more
classical setting of monogenic functions, see for example the
book \cite{DSS}. In the treatment in \cite{DSS} no tensor products of Hilbert Clifford
modules are involved. In the framework of slice hyperholomorphic
analysis  we have already introduced and studied the Hardy spaces
(see \cite{acs2,abcsNP,acls}), and Bergman spaces (see
\cite{CGCS1,INDAM,CGCS2}).  Here we begin the study  of the main
properties of the Fock spaces in the slice hyperholomorphic
setting.\smallskip

We start by recalling the definition of the Fock space in the classical complex analysis case
(for the origins of the theory see \cite{Fock}).\\
For $n\in \mathbb{N}$ let $z=(z_1,\ldots,z_n)\in \mathbb{C}^n$ where $z_j=x_j+iy_j$, $x_j$,
$y_j\in \mathbb{R}$  ($j=1,...,n$) and denote by
$$
d\mu(z):=\pi^{-n} \Pi_{j=1}^ndx_jdy_j
$$
the normalized Lebesgue measure on $\mathbb{C}^n$. The Fock space of holomorphic functions $f$
 defined on $\mathbb{C}^n$ is
\begin{equation}
\label{fockgeometric}
 F_n:=\{\ f:\mathbb{C}^n\mapsto \mathbb{C}\ {\rm such \ that}\ \
 \int_{\mathbb{C}^n} |f(z)|^2e^{-|z|^2}d\mu(z)<\infty \ \}.
 \end{equation}
 The space $F_n$ with the scalar product
 $$
 \langle f,g\rangle_{F_n}= \int_{\mathbb{C}^n} f(z) \overline{g(z)}e^{-|z|^2}d\mu(z)
 $$
 becomes a Hilbert space and the norm is
 $$
 \|f\|^2_{F_n}=\int_{\mathbb{C}^n} |f(z)|^2e^{-|z|^2}d\mu(z),\ \ \ f\in F_n.
 $$
The space $F_n$ is called boson Fock space and since we will treat this case in the sequel we will
 refer to it simply as  Fock space.
One of its most important properties is that it is a reproducing kernel Hilbert space.
If we denote by $\langle\cdot,\cdot\rangle_{\mathbb{C}^n}$ the natural scalar product in $\mathbb{C}^n$ defined by
$\langle u,v\rangle_{\mathbb{C}^n}:=\sum_{j=1}^n u_j\overline{v}_j$, for every $u$,
$v\in \mathbb{C}^n$ we define the function
\begin{equation}
\label{enuv}
\psi_u(z)=e^{\langle u,v\rangle_{\mathbb{C}^n}}=e^{\sum_{j=1}^n
u_j\overline{v}_j}.
\end{equation}
We have the reproducing property
$$
\langle f,\psi_u\rangle_{F_n}=f(u),\ \ \ \ {\rm for\ all}\ \ \ f\in F_n.
$$
So there are two equivalent characterizations of the Fock space $F_n$; one geometric, in terms of
integrals (see \eqref{fockgeometric}), and one analytic, obtained the reproducing kernel property (or,
directly from \eqref{fockgeometric}): an entire function
$f(z)=\sum_{m\in\mathbb N_0^n}a_mz^m$ of $n$ complex variables $z=(z_1,\ldots, z_n)$ is in $F_n$ if and only if its Taylor
coefficients satisfy
\[
\sum_{m\in\mathbb N_0^n}m!|a_m|^2<\infty,
\]
where we have used the multi-index notation. A third characterization is of importance, namely (with appropriate
identification, and with $\circ$ denoting the symmetric tensor product)
\[
F_n=\oplus_{k=0}^\infty (\mathbb C^n)^{\circ k}.
\]
In this paper we will address some aspects of these three
characterizations in the quaternionic and Clifford algebras settings.\\

The paper consists of four sections besides the introduction. In
Section \ref{sec2} we give a brief survey of infinite dimensional
analysis. In Section \ref{sec3} we study the quaternionic Fock
space in one quaternionic variable. We then discuss, in Section
\ref{sec4}, the full Fock space. In order to define it, we need to
study tensor products of quaternionic two-sided Hilbert spaces.
Tensor product of quaternionic vector spaces have been treated in
the literature at various level, see e.g. \cite{ak},
\cite{MR1352898,MR768240}. This section in particular opens the
way to study a quaternionic infinite dimensional analysis. The
last section considers the case of slice monogenic functions.

\section{A brief survey of infinite dimensional analysis}
\setcounter{equation}{0}
\label{sec2}

There are various ways to introduce infinite dimensional
analysis. We mention here four related approaches:\\

\noindent 1. \underline{The white noise space and the Bochner-Minlos
theorem:} The formula
\begin{equation}
e^{-\frac{t^2}{2}}=\frac{1}{\sqrt{2\pi}}\int_{\mathbb
R}e^{-\frac{u^2}{2}}e^{-itu}du
\end{equation}
is an illustration of Bochner's theorem. It is well known that
there is no such formula when $\mathbb R$ is replaced by an
infinite dimensional Hilbert space. On the other hand, the
Bochner-Minlos theorem asserts that there exists a probability
measure $P$ on the space $\mathcal S^\prime$ of real tempered
distributions such that
\begin{equation}
\label{gaussian}
e^{-\frac{\|s\|_2^2}{2}}=\int_{\mathcal
S^\prime}e^{i\langle s^\prime, s\rangle}dP(s^\prime).
\end{equation}
In this expression, $s$ is belongs to the space $\mathcal S$ of
real-valued Schwartz function, the duality between $\mathcal S$
and $\mathcal S^\prime$ is denoted by $\langle s^\prime,s
\rangle$ and $\|\cdot\|_2$ denotes the $\mathbf L_2(\mathbb R,
dx)$ norm.\smallskip

The probability space $\mathbf L_2(\mathcal S^\prime, P)$ is
called the white noise space, and is denoted  by $\mathcal W$.
Denoting by $Q_s$ the map $s^\prime\mapsto \langle s^\prime,
s\rangle$ we see that \eqref{gaussian} induces an isometry, which
we denote $Q_f$, from the Lebesgue space $\mathbf L_2(\mathbb R,
dx)$ into the white noise space. We now give an important family
of orthogonal basis $\left(H_\alpha, \ \alpha \in \ell \right)$ of
the white noise space, indexed by the set $\ell$ of sequences
$(\alpha_1,\alpha_2,\ldots)$, with entries in
\[
\mathbb N_0=\left\{ 0,1,2,3,\ldots\right\},
\]
where $\alpha_k\not =0$ for only a finite number of indices $k$.
Let $h_0,h_1,\ldots$ denote the Hermite polynomials, and let
$\xi_1,\xi_2,\ldots$ be an orthonormal basis of $\mathbf
L_2(\mathbb R, dx)$ (typically, the Hermite functions, but other
choices are possible). Then

\begin{equation}
H_\alpha = \prod _ {k=1} ^\infty {h_{\alpha _k} (Q_{{\xi}_k} )},
\end{equation}
and, with the multi-index notation
\[
\alpha!=\alpha_1!\alpha_2!\cdots,
\]
we have
\begin{equation}
\label{eq:fock1}
\|H_\alpha\|_{\mathcal W}^2=\alpha!.
\end{equation}
The decomposition of an element $f\in \mathcal W$ along the basis
$(H_\alpha)_{\alpha \in \ell}$ is called the chaos expansion.\\

\noindent 2. \underline{The Bargmann space in infinitely many variables:}
When in \eqref{enuv}, $\mathbb C^n$ is replaced by $\ell_2(\mathbb
N)$, we have the function
\begin{equation}
\label{euvl2}
\psi_u(z)=e^{\langle u,v\rangle_{\ell_2(\mathbb
N)}}=e^{\sum_{j=1}^\infty u_j\overline{v}_j}.
\end{equation}
The map $H_\alpha\mapsto z^\alpha$ is called the Hermite
transform, and is unitary from the white noise space onto the
reproducing kernel Hilbert space with reproducing kernel
\eqref{euvl2}.\\

\noindent 3. \underline{The Fock space:} We denote by $\circ$ the symmetrized tensor product and
by
\[
\Gamma^\circ(\bigH)=\oplus_{n=0}^\infty\mathcal H^{\circ n},
\]
the symmetric Fock space associated to a Hilbert space $\mathcal H$.
Then, $\Gamma^\circ( \mathbf L_2(\mathbb R, dx))$ can be identified with the white noise space via the
Wiener-It\^{o}-Segal transform defined as follows (see \cite[p. 165]{MR1851117}):
\[
\xi_\alpha=\xi_{i_1}^{\circ \alpha_{i_1}}\circ \cdots \circ \xi_{i_m}^{\circ \alpha_{i_m}} \in \bigH^{\circ n}\quad\mapsto H_\alpha
\]
This is the starting point of our approach to quaternionic
infinite dimensional analysis; see Section \ref{sec4}.\\

\noindent 4. \underline{The free setting. The full Fock space:}
It is defined by
\[
\Gamma(\bigH)=\oplus_{n=0}^\infty\mathcal H^{\otimes n},
\]
and allows to develop the free analog of the white noise space theory. See \cite{MR1217253,MR2227827} for background for the
free setting. See \cite{MR3038506} for recent applications to the theory of non commutative stochastic distributions.\\

We refer in particular to the papers
\cite{MR0076317,MR0077908,MR0157250,bargmann} and the books
\cite{MR1244577,MR2444857,HiHi93,MR1408433,MR1851117,MR1746976} for more information on these various aspects.

\section{The Fock space in the slice regular case}
\setcounter{equation}{0}
\label{sec3}

 The algebra of quaternions is indicated by the symbol $\mathbb{H}$.
The imaginary units  in $\mathbb{H}$ are denoted by $i$, $j$ and $k$, respectively,
 and an element in $\mathbb{H}$ is of the form $q=x_0+ix_1+jx_2+kx_3$, for $x_\ell\in \mathbb{R}$.
The real part, the imaginary part and the modulus of a quaternion are defined as
${\rm Re}(q)=x_0$, ${\rm Im}(q)=i x_1 +j x_2 +k x_3$, $|q|^2=x_0^2+x_1^2+x_2^2+x_3^2$, respectively.
The conjugate of the quaternion $q=x_0+ix_1+jx_2+kx_3$ is defined by
$\bar q={\rm Re }(q)-{\rm Im }(q)=x_0-i x_1-j x_2-k x_3$
 and it satisfies
$$
|q|^2=q\bar q=\bar q q.
$$
The unit sphere of purely imaginary quaternions is
$$
\mathbb{S}=\{q=ix_1+jx_2+kx_3\ {\rm such \ that}\
x_1^2+x_2^2+x_3^2=1\}.
$$
Notice that if $I\in\mathbb{S}$, then
$I^2=-1$; for this reason the elements of $\mathbb{S}$ are also called
imaginary units. Note that $\mathbb{S}$ is a 2-dimensional sphere in $\mathbb{R}^4$.
Given a nonreal quaternion $q=x_0+{\rm Im} (q)=x_0+I |{\rm Im} (q)|$, $I={\rm Im} (q)/|{\rm Im} (q)|\in\mathbb{S}$, we can associate to it the 2-dimensional sphere defined by
$$
[q]=\{x_0+I  {\rm Im} (q)|\  :\  \ I\in\mathbb{S}\}.
$$
This sphere has center at the real point $x_0$ and radius $|{\rm Im} (q)|$.
An element in the complex plane $\mathbb{C}_I=\mathbb{R}+I\mathbb{R}$ is denoted by $x+Iy$.

\begin{defn}[Slice regular (or slice hyperholomorphic) functions] Let $U$ be an open set in
$\mathbb{H}$ and consider  a real differentiable function $f:U \to \mathbb{H}$.
Denote by  $f_I$ the restriction of $f$  to the complex plane $\mathbb{C}_I$.
\\
The function  $f$  is (left) slice regular (or (left) slice hyperholomorphic) if, for every $I \in
\mathbb{S}$, it satisfies:
$$\overline{\partial}_If_I(x+Iy):=\frac{1}{2}\left(\frac{\partial}{\partial x}
+I\frac{\partial}{\partial y}\right)f_I(x+Iy)=0,$$
on $U \cap \mathbb{C}_I$. The set of (left) slice regular functions on $U$ will be denoted by $\mathcal{R}(U)$.
\\
The function  $f$ is
right slice regular (or right slice hyperholomorphic)  if, for every $I \in
\mathbb{S}$, it satisfies:
$$(f_I\overline{\partial}_I)(x+Iy):=\frac{1}{2}\left(\frac{\partial}{\partial x}f_I(x+Iy)
+\frac{\partial}{\partial y}f_I(x+Iy)I\right)=0,$$
 on $U \cap \mathbb{C}_I$.
\end{defn}
The class of slice hyperholomorphic quaternionic valued functions is important since power series centered at real points are slice hyperholomorphic:
if  $B=B(y_0,R)$ is  the open ball centered at the real point $y_0$ and radius $R>0$ and if
$f:\,B \to \mathbb{H}$ is  a left slice regular function then $f$ admits the power series expansion
$$
f(q)=\sum_{m=0}^{+\infty} (q-y_0)^m\frac{1}{m!}\frac{\partial^m f}{\partial x^m}(y_0),
$$
converging on $B$.

A main property of the slice hyperholomorphic functions is the so-called Representation Formula (or Structure Formula).
It holds on a particular class of open sets which are described below.
\begin{defn}[Axially symmetric domain]\label{def_axially}
Let $U \subseteq \mathbb{H}$. We say that $U$ is
\textnormal{axially symmetric} if, for all $x+Iy \in U$, the whole
2-sphere $[x+Iy]$ is contained in $U$.
\end{defn}

\begin{defn}[Slice domain]\label{defs-domain}
Let $U \subseteq \mathbb{H}$ be a domain in $\mathbb{H}$. We say that $U$ is a
\textnormal{slice domain (s-domain for short)} if $U \cap \mathbb{R}$ is non empty and
if $ U\cap\mathbb{C}_I$ is a domain in $\mathbb{C}_I$ for all $I \in \mathbb{S}$.
\end{defn}

\begin{thm}[Representation Formula]\label{formulaQUATER}
Let $U$ be an axially symmetric s-domain $U \subseteq  \mathbb{H}$.

Let $f$ be a (left) slice regular function on $U$.  Choose any
$J\in \mathbb{S}$.  Then the following equality holds for all $q=x+yI \in U $:
\begin{equation}\label{distribution}
f(x+Iy) =\frac{1}{2}\Big[   f(x+Jy)+f(x-Jy)\Big] +I\frac{1}{2}\Big[ J[f(x-Jy)-f(x+Jy)]\Big].
\end{equation}
\end{thm}
\begin{rem}{\rm One of the applications of the Representation Formula is the fact that any function defined on an open set $\Omega_I$ of a complex plane
$\mathbb C_I$ which belongs to the kernel of the Cauchy-Riemann operator can be uniquely extended to a slice hyperholomorphic function defined on the axially symmetric completion of $\Omega_I$ (see \cite{bookfunctional}).
}
\end{rem}
We now define the Fock space in this framework.
\begin{defn}[Slice hyperholomorphic quaternionic Fock space]\label{Fock_S-B}
Let $I$ be any element in $\mathbb{S}$. Consider the set
$$
\mathcal{F}(\H)=\{f\in\mathcal{R}(\H)\ |\  \int_{\C_I} e^{-|p|^2} |f_I(p)|^2 d\sigma(x,y)<\infty\}
$$
where $p=x+Iy$, $d\sigma(x,y):=\frac{1}{\pi}dxdy$.
We will call $\mathcal{F}(\H)$ (slice hyperholomorphic) quaternionic Fock space.
\end{defn}
We endow $\mathcal{F}(\H)$ with the inner product
\begin{equation}\label{inner}
\langle f, g\rangle:=\int_{\C_I} e^{-|p|^2} \overline{g_I(p)}f_I(p) d\sigma(x,y);
\end{equation}
we will show below that this definition, as well as the definition of Fock space, do not depend on the imaginary unit $I\in\mathbb{S}$.\\
The norm induced by the inner product is then
$$
\|f\|^2=\int_{\C_I} e^{-|p|^2} |f_I(p)|^2 d\sigma(x,y).
$$

We have the following result:
\begin{prop}
The quaternionic Fock space $\mathcal{F}(\H)$ contains the monomials $p^n$, $n\in\mathbb{N}$ which form an orthogonal basis.
\end{prop}
\begin{proof}
Let us choose an imaginary unit $I\in\mathbb{S}$ and, for $n, m\in \mathbb{N}$, compute
$$
\langle p^n,p^m\rangle=\int_{\C_I} e^{-|p|^2} \overline{p^m}p^n d\sigma(x,y).
$$
By using polar coordinates, we write $p=\rho e^{I\theta}$ and we have
\[
\begin{split}
\langle p^n,p^m\rangle&=\frac{1}{\pi}\int_{0}^{2\pi}\int_0^{+\infty} e^{-\rho^2} {\rho^m} e^{-Im\theta}{\rho^n} e^{In\theta} \rho\, d\rho\,  d\theta\\
&=\frac{1}{\pi}\int_{0}^{2\pi}\int_0^{+\infty} e^{-\rho^2} \rho^{m+n+1} e^{I(n-m)\theta} d\rho\, d\theta\\
&=\frac{1}{2\pi}\int_{0}^{2\pi}e^{I(n-m)\theta}d\theta\int_0^{+\infty} e^{-\rho^2} \rho^{m+n}  d\rho^2. \\
\end{split}
\]
Since $\int_{0}^{2\pi}e^{I(n-m)\theta}d\theta$ vanishes for $n\not=m$ and equals $2\pi$ for $n=m$, we have
$\langle p^n,p^m\rangle =0$ for $n\not=m$. For $n=m$, standard computations give
\[
\langle p^n,p^n\rangle =\int_0^{+\infty} e^{-\rho^2} \rho^{2n}  d\rho^2=n!.
\]
Thus the monomials $p^n$ belong to $\mathcal{F}(\H)$ and any two of them are orthogonal. We now show that these monomials form a basis for
$\mathcal{F}(\H)$.
A function $f\in\mathcal{F}(\H)$ is entire so it admits series expansion of the form $f(p)=\sum_{m=0}^{+\infty} p^m a_m$ and thus the monomials
$p^n$ are generators. To show that they are independent, we show that if $\langle f, p^n\rangle=0$ for all $n\in\mathbb{N}$ then $f$ is identically zero.
We have:
\[
\begin{split}
\langle f, p^n \rangle &= \langle \sum_{m=0}^{+\infty} p^m a_m , p^n \rangle\\
&=\int_{\C_I} e^{-|p|^2} \overline{p^n}\left(\sum_{m=0}^{+\infty} p^m a_m\right)  d\sigma(x,y)\\
\end{split}
\]
and so
\[
\begin{split}
\langle f, p^n \rangle &=\lim_{r\to+\infty}\int_{|p|<r,\, p\in\C_I} e^{-|p|^2} \overline{p^n} \left(\sum_{m=0}^{+\infty} p^m a_m\right)  d\sigma(x,y)\\
&=\lim_{r\to+\infty}\sum_{m=0}^{+\infty}  \left(\int_{|p|<r,\, p\in\C_I} e^{-|p|^2} \overline{p^n} p^m  d\sigma(x,y) \right) a_m\\
&=\lim_{r\to+\infty}\int_{0}^r \rho^{2n} e^{-r^2} dr^2 a_n=n! a_n,
\end{split}
\]
where we used the fact that the series expansion converges uniformly on $|p|<r$, thus we can exchange the series with the integration where needed.
Thus $\langle f, p^n \rangle=0$ for all $n$ if and only if $a_n=0$, i.e. $f\equiv 0$.
\end{proof}
\begin{prop}\label{independence}
The definition of inner product (\ref{inner}) does not depend on the imaginary unit $I\in\mathbb{S}$.
\end{prop}
\begin{proof}
Let $f(p)=\sum_{m=0}^{+\infty} p^m a_m$, $g(p)=\sum_{m=0}^{+\infty} p^m b_m\in \mathcal{F}(\H)$ and let $I\in\mathbb{S}$.
We have
\[
\begin{split}
\langle f, g \rangle &= \langle \sum_{m=0}^{+\infty} p^m a_m ,  \sum_{m=0}^{+\infty} p^m b_m\rangle\\
&=\int_{\C_I} e^{-|p|^2} \left(\overline{\sum_{m=0}^{+\infty} p^m a_m} \right) \left(  \sum_{n=0}^{+\infty} p^n b_n \right)d\sigma(x,y)\\
&=\int_{\C_I} e^{-|p|^2} \left( \sum_{m=0}^{+\infty}  \bar{a}_m  \bar{p}^m \right) \left( \sum_{n=0}^{+\infty} p^n b_n\right) d\sigma(x,y)\\
\end{split}
\]
so that
\[
\begin{split}
\langle f, g \rangle &=\sum_{n=0}^{+\infty} \int_{\C_I} e^{-|p|^2}  \bar{a}_n  \bar{p}^n   p^n b_n d\sigma(x,y)\\
&=\sum_{n=0}^{+\infty} \bar{a}_n \left(\int_{\C_I} e^{-|p|^2}   \bar{p}^n   p^n  d\sigma(x,y)\right) b_n\\
&=\sum_{n=0}^{+\infty} n!  \bar{a}_n   b_n, \\
\end{split}
\]
which shows that the computation does not depend on the chosen imaginary unit $I$.
\end{proof}
Let us recall that the slice regular exponential function is defined by
$$
e^p:=\sum_{n=0}^{+\infty} \frac{p^n}{n!}.
$$
We need to generalize the definition of the function
$e^{zw}=\sum_{m=0}^{+\infty} \frac{(zw)^m}{m!}$, $z,w\in\C$ to the slice hyperholomorphic setting.

We first observe that if we set $e^{pq}=\sum_{m=0}^{+\infty} \frac{(pq)^m}{m!}$
then the function $e^{pq}$ does not have any good property of regularity:
it is not slice regular neither in $p$ nor in $q$ (while $e^{zw}$ is holomorphic in both the variables).
Let us consider $p$ as a variable and $q$ as a parameter and set:
\begin{equation}\label{exp}
e_\star^{pq}=\sum_{n=0}^{+\infty} \frac{(pq)^{\star n}}{n!}=\sum_{n=0}^{+\infty} \frac{p^n q^n}{n!}
\end{equation}
where the $\star$-product (see \cite{bookfunctional}) is computed with respect to the variable $p$. It is immediate that $e_\star^{pq}$ is a function left slice regular in $p$ and right regular in $q$.
\begin{rem}{\rm
The definition (\ref{exp}) is consistent with the fact that we are looking for a slice regular extension of $e^{zw}$.
In fact, we start from the function $e^{zw}=\sum_{n=0}^{+\infty} \frac{z^nw^n}{n!}$, which is holomorphic in $z$ seen as an element on the complex plane $\C_I$;
we then use the Representation Formula to get the extension to $\mathbb H$:
$$
{\rm ext}(e^{zw})=\frac 12 ( 1- I_qI ) \sum_{n=0}^{+\infty} \frac{z^nw^n}{n!}+ \frac 12 ( 1+ I_qI )  \sum_{n=0}^{+\infty}\frac{{\bar z}^nw^n}{n!}=e^{qw}
$$
and since $w$ is arbitrary, we get the statement.}
\end{rem}
We now set $k_q(p):=e_\star^{p\bar q}$ and we discuss the reproducing property in the Fock space.
\begin{thm}
For every $f\in\mathcal{F}(\H)$ we have
$$
\langle f, k_q \rangle= f(q).
$$
Moreover, $\langle k_q,k_s\rangle= e_\star^{s\bar{q}}$.
\end{thm}
\begin{proof}
We have
\[
\begin{split}
\langle f, k_q \rangle &= \int_{\C_I} e^{-|p|^2} \overline{ e^{p\bar q}} f(p) d\sigma(x,y)\\
&= \int_{\C_I} e^{-|p|^2} \left(\sum_{n=0}^{+\infty} \frac{q^n\bar p^n}{n!} \right)\left(\sum_{m=0}^{+\infty} p^ma_m \right)d\sigma(x,y)\\
&= \sum_{n=0}^{+\infty}\sum_{m=0}^{+\infty} \frac{q^n}{n!}\langle p^m, p^n\rangle a_m
\\
&
=\sum_{n=0}^{+\infty} q^n a_n
\\
&
=f(q).
\end{split}
\]
Similarly, we have
\[
\begin{split}
\langle k_q,k_s\rangle &= \int_{\C_I} e^{-|p|^2} \overline{ e^{p\bar s}}  e^{p\bar q} d\sigma(x,y)\\
&= \int_{\C_I} e^{-|p|^2} \left(\sum_{n=0}^{+\infty} \frac{s^n\bar p^n}{n!} \right) \left(\sum_{m=0}^{+\infty} p^m\bar{q}^m\right) d\sigma(x,y)\\
&= \sum_{n=0}^{+\infty}\sum_{m=0}^{+\infty} \frac{s^n}{n!}\langle p^m, p^n\rangle \frac{\bar{q}^m}{m!}
\\
&
=\sum_{n=0}^{+\infty} \frac{s^n \bar{q}^n}{n!}
\\
&
=e_\star^{s\bar q}.
\end{split}
\]
\end{proof}
\begin{prop}\label{Fock1}
A function $f(p)=\sum_{m=0}^{+\infty} p^m a_m$ belongs to $\mathcal{F}(\H)$ if and only if $\sum_{m=0}^{+\infty}  |a_m|^2 m!<\infty$.
\end{prop}
\begin{proof}
Let us use polar coordinates; with computations similar to those in the proof of Proposition \ref{independence} and using the Parseval identity, we have
\[
\begin{split}
\int_{\C_I} e^{-|p|^2} |f(p)|^2 d\sigma(x,y)&=\lim_{r\to +\infty} \frac{1}{\pi}\int_0^r  e^{-\rho^2}\int_0^{2\pi} |f(\rho e^{I\theta})|^2 \rho\, d\theta\, d\rho \\
&=2\lim_{r\to +\infty}  \int_0^r  e^{-\rho^2}(\sum_{m=0}^{+\infty} \rho^{2m}|a_m|^2) \rho\,  d\rho \\
&=2\lim_{r\to +\infty}  \sum_{m=0}^{+\infty}\int_0^r  e^{-\rho^2} \rho^{2m+1}|a_m|^2  d\rho
\\
&= \sum_{m=0}^{+\infty} |a_m|^2m!\\
\end{split}
\]
and the statement follows.
\end{proof}

\section{Quaternion full Fock space and symmetric Fock space}
\setcounter{equation}{0}
\label{sec4}

Let $V$ be a right vector space over $\mathbb H$. Recall that a quaternionic inner product on $V$ is a map $\langle \cdot,\cdot \rangle : V \times V\to \mathbb H$ satisfies the same properties of a complex inner product, with the exception of
the homogeneity requirement which is replaced by
\[
\langle  u \alpha ,  v \beta \rangle = \overline{\beta} \langle u ,  v\rangle \alpha,
\]
and that if $V$ is complete with respect to the norm induced by the inner product, it is called a right quaternionic Hilbert space. A similar definition can be given in the case of a quaternionic vector space on the left or two-sided.\\

Let $\bigH$ be a two-sided quaternionic Hilbert space. Then one may consider the
quaternionic $n$-fold Hilbert space tensor power $\bigH^{\otimes n}$ defined by
\[
\bigH^{\otimes n}=\bigH\otimes \bigH \otimes \cdots \otimes \bigH \quad \text{($n$ times)},
\]
where all tensor product are over $\mathbb H$.

\begin{rem}{\rm
A convenient way of constructing $\bigH^{\otimes n}$ is
inductively. Recall that if $M$ is a left $R$-module and $N$ is a
right $R$-module, then a tensor product of them ${M_R} \otimes
{_RN}$ is an abelian group together with a bilinear map $\delta:
M \times N \to {M_R} \otimes {_RN}$ which is universal in the
sense that for any abelian group $A$ and a bilinear map $f: M
\times N \to A$, there is a unique group homomorphism $\tilde f :
{M_R} \otimes {_RN} \to A$ such that $\tilde f \otimes \delta
=f$. If furthermore, $M$ is a right $S$-module and $N$ is a left
$T$-module, then ${_SM_R} \otimes {_RN_T}$ is a $(S,T)$-bi-module if
one defines $sz=(s \otimes 1)z$ and $zt=z(1 \otimes t)$ for $z
\in {_SM_R} \otimes {_RN_T}$. Since it holds that
\[
({_RM_S} \otimes {_SN_T}) \otimes {_TP_U} \cong {_RM_S} \otimes ({_SN_T} \otimes {_TP_U}),
\]
one can define inductively the tensor product of $M_1, \dots ,M_n$, where $M_i$ is a $(R_{i-1},R_i)$-bi-module, and obtain a $(R_0,R_n)$-bi-module,
\[
{_{R_0}{M_1}_{R_1}} \otimes {_{R_1}{M_2}_{R_2}} \otimes \cdots \otimes {_{R_{n-1}}{M_1}_{R_n}}.
\]
For more details see \cite[pp. 133-135]{Jacobson}.
One can also define it non-inductively (see \cite[pp. 264]{Bourbaki}).}
\end{rem}

We make the convention $\bigH^{\otimes 0}=\mathbb \mathbb H$, and
the element $1 \in \mathbb \mathbb H$ is called the vacuum vector
and denoted by $\bf 1$. For the case of two Hilbert spaces in the
next proposition, see also \cite[equation (3)]{ak}.
\begin{prop}\label{prop 4.2}
Let $\langle \cdot, \cdot \rangle$ be the inner product of $\bigH$, and assume that it satisfies also the additional property
\[
\langle u , \lambda v\rangle =\langle \overline{\lambda} u ,  v\rangle.
\]
Then, it induces an inner product on
 $\bigH^{\otimes n}$,
\[
\langle u_1 \otimes \cdots \otimes u_n , v_1 \otimes \cdots
\otimes v_n \rangle =\langle \langle\cdots\langle \langle \langle u_1,v_1\rangle u_2, v_2 \rangle u_3, v_3 \rangle  \cdots \rangle u_n, v_n \rangle,
\]
with the same additional property.
\end{prop}
\begin{proof}[Proof]
The statement clearly holds for $n=1$.
By induction,
\[
\begin{split}
\langle u_1 \otimes \cdots \otimes u_n  \alpha, v_1 \otimes \cdots \otimes v_n \beta \rangle
&=\langle \langle u_1 \otimes \cdots \otimes u_{n-1} , v_1 \otimes \cdots \otimes v_{n-1} \rangle u_n \alpha , v_n \beta \rangle\\
&=\overline{\beta} \langle \langle u_1 \otimes \cdots \otimes u_{n-1} , v_1 \otimes \cdots \otimes v_{n-1} \rangle u_n  , v_n \rangle \alpha\\
&=\overline{\beta}\langle u_1 \otimes \cdots \otimes u_n  , v_1 \otimes \cdots \otimes v_n  \rangle \alpha,
\end{split}
\]
and
\[
\begin{split}
\overline{\langle  v_1 \otimes \cdots \otimes v_n ,u_1 \otimes \cdots \otimes u_n   \rangle}
&=\overline{\langle \langle v_1 \otimes \cdots \otimes v_{n-1} , u_1 \otimes \cdots \otimes u_{n-1} \rangle v_n  ,u_n  \rangle}\\
&=\langle u_n , \langle v_1 \otimes \cdots \otimes v_{n-1} , u_1 \otimes \cdots \otimes u_{n-1} \rangle v_n   \rangle\\
&=\langle \overline{\langle v_1 \otimes \cdots \otimes v_{n-1} , u_1 \otimes \cdots \otimes u_{n-1} \rangle} u_n ,  v_n   \rangle\\
&=\langle \langle  u_1 \otimes \cdots \otimes u_{n-1},v_1 \otimes \cdots \otimes v_{n-1}  \rangle u_n ,  v_n   \rangle\\
&=\langle u_1 \otimes \cdots \otimes u_n , v_1 \otimes \cdots \otimes v_n \rangle.
\end{split}
\]
For the additional property, we obtain
\[
\begin{split}
\langle u_1 \otimes \cdots \otimes u_n  , \lambda v_1 \otimes \cdots \otimes v_n  \rangle
&=\langle \langle u_1 \otimes \cdots \otimes u_{n-1} ,\lambda  v_1 \otimes \cdots \otimes v_{n-1} \rangle u_n  , v_n  \rangle\\
&= \langle \langle  \overline{\lambda} u_1 \otimes \cdots \otimes u_{n-1} , v_1 \otimes \cdots \otimes v_{n-1} \rangle u_n  , v_n \rangle \\
&=\langle \overline{\lambda} u_1 \otimes \cdots \otimes u_n  ,  v_1 \otimes \cdots \otimes v_n  \rangle.
\end{split}
\]
Additivity and positivity are obvious.
\end{proof}
\begin{defn}
The quaternionic full Fock module over an Hilbert space $\bigH$ is the space
\[
\bigF(\bigH)=\oplus_{n=0}^\infty\bigH^{\otimes n},
\]
with the corresponding inner product.\\
\end{defn}

\begin{defn}
Let $u \in \bigH$. The right-linear map $T_u: \bigF(\bigH) \to \bigF(\bigH)$ defined by
\[
T_u(u_1 \otimes \dots \otimes u_n)=u \otimes u_1 \otimes \dots \otimes u_n,
\]
is called the creation map.
The right-linear map $T_u^*: \bigF(\bigH) \to \bigF(\bigH)$ defined by
\[
T_u^*(u_0 \otimes \dots \otimes u_n)=\overline{\langle u ,u_0 \rangle}  u_1 \otimes \dots \otimes u_n,
\]
is called the annihilator map.
\end{defn}
The following result is the quaternionic counterpart of a classical result:
\begin{prop}
$T_u^*$ is the adjoint of $T_u$.
\end{prop}
\begin{proof}[Proof]
The statement follows from
\[
\begin{split}
\langle T_u^*(u_0  \otimes \dots \otimes u_n), v_1 \otimes \dots \otimes v_n\rangle
&=\langle \overline{\langle u,u_0 \rangle}  u_1 \otimes \dots \otimes u_n, v_1 \otimes \dots \otimes v_n\rangle\\
&=\langle \langle u_0,u \rangle  u_1 \otimes \dots \otimes u_n, v_1 \otimes \dots \otimes v_n\rangle\\
&=\langle u_0  \otimes \dots \otimes u_n, u \otimes v_1 \otimes \dots \otimes v_n\rangle\\
&=\langle u_0  \otimes \dots \otimes u_n, T_u (v_1 \otimes \dots \otimes v_n)\rangle
\end{split}
\]
\end{proof}

\begin{rem}{
Note that the isometry $u \mapsto T_u$ is both left-linear and right-linear.}
\end{rem}

The complex-valued version of the following proposition appears
in \cite{MR2540072,MR2770019}, where the free Brownian motion is
defined and studied. The derivative of the function $X(t)$ is
studied in \cite{2013arXiv1311.3239A}.

\begin{prop}[The non-symmetric quaternionic Brownian motion]
Let $\bigH=\mathbf L_2(\mathbb R^+, dx)$, and consider
$X(t)=T_{{\bf 1}_{[0,t]}}+T_{{\bf 1}_{[0,t]}}^*$. Then
\[
\langle X(t) {\bf1}, X(s) {\bf1} \rangle = \min \{t,s \}.
\]
In particular $X(t)$ is self-adjoint, and if one consider the
expectation $E: B(\bigF(\bigH)) \to \mathbb H$  defined by
$E(T)=\langle T {\bf 1}, {\bf 1} \rangle$, then
\[
E(X(s)^*X(t))=\min\{t,s\}.
\]
\end{prop}

\begin{proof}[Proof]
More generally, note that
\[
\langle (T_u+T_u^*) {\bf1}, (T_u+T_u^*) {\bf1} \rangle
=\langle T_u {\bf1}, T_u {\bf 1} \rangle
=\langle u, u \rangle.
\]
Since $\langle {\bf 1}_{[0,t]},{\bf 1}_{[0,s]}
\rangle_{\bigH}=\min \{t,s\}$, the result follows.
\end{proof}

The symmetric product $\circ$ is defined by
\[
u_1 \circ \cdots \circ u_n=\frac{1}{n!} \sum_{\sigma \in S_n}
u_{\sigma(1)}\otimes \cdots \otimes u_{\sigma(n)},
\]
and the closed subspace of $\bigH^{\otimes n}$ generated by all
vectors of this form is called the $n$-th symmetric power of
$\bigH$, and denoted by $\bigH^{\circ n}$.

\begin{prop}\label{tensorInner}
Let $\langle \cdot, \cdot \rangle$ be the inner product of $\bigH$, and assume that it satisfies also the additional property
\[
\langle u , \lambda v\rangle =\langle \overline{\lambda} u ,  v\rangle.
\]
Then, it induces an inner product on
 $\bigH^{\circ n}$,
\[
\begin{split}
\langle u_1 \circ \cdots & \circ u_n , v_1  \circ \cdots
\circ v_n \rangle \\
&=\frac{1}{n!^2}\sum_{\sigma,\tau \in S_n}\langle \langle\cdots\langle \langle \langle u_{\sigma(1)},v_{\tau(1)}\rangle u_{\sigma(2)}, v_{\tau(2)} \rangle u_{\sigma(3)}, v_{\tau(3)} \rangle  \cdots \rangle u_{\sigma(n)}, v_{\tau(n)} \rangle,
\end{split}
\]
with the same additional property.
\end{prop}
\begin{proof} The result follows as in the proof of Proposition \ref{prop 4.2}.
\end{proof}
In the classical case (where $\bigH$ is a Hilbert space over the field $\mathbb R$ or $\mathbb C$),
another natural inner-product is usually being used, namely the symmetric inner product. It is defined by
\[
\langle u_1 \circ \cdots \circ u_n , v_1  \circ \cdots\circ v_n \rangle  =per\left(\langle u_i, v_j \rangle\right),
\]
where $per(A)$ is called the permanent of $A$ and has the same definition
as a determinant, with the exception that the factor $sgn(\sigma)$ is omitted.
An easy computation implies that when restricted to the $n$-fold symmetric tensor power $\bigH^{\circ n}$,
the second inner product (i.e. the symmetric inner product) is simply $n!$ times the first inner product
(the one which is defined in Proposition \ref{tensorInner}). This gives rise to the following definition.
\begin{defn}
Let $\langle \cdot, \cdot \rangle$ be the inner product of $\bigH$, and assume that it satisfies also the additional property
\[
\langle u , \lambda v\rangle =\langle \overline{\lambda} u ,  v\rangle.
\]
Then, the symmetric inner product on
 $\bigH^{\circ n}$ is defined by,
\[
\begin{split}
\langle u_1 \circ \cdots & \circ u_n , v_1  \circ \cdots
\circ v_n \rangle \\
&=\frac{1}{n!}\sum_{\sigma,\tau \in S_n}\langle \langle\cdots\langle \langle \langle u_{\sigma(1)},v_{\tau(1)}\rangle u_{\sigma(2)}, v_{\tau(2)} \rangle u_{\sigma(3)}, v_{\tau(3)} \rangle  \cdots \rangle u_{\sigma(n)}, v_{\tau(n)} \rangle.
\end{split}
\]
\end{defn}

We now focus on the special case of the symmetric Fock space $\bigF^\circ(\bigH)$ where $p$ is a quaternion variable and $\bigH=p \mathbb H$. When no confusion can arise, we will simply denote it by $\bigF^\circ(\mathbb H)$. The following result shows the relation with the Fock space as introduced in Definition \ref{Fock_S-B}, see Proposition \ref{Fock1}.

\begin{prop}\label{seriesFock}
$\bigF^\circ(\mathbb H)$ is the space of all entire functions
\[
\sum_{n=0}^\infty p^n a_n
\]
satisfying
$\sum_{n=0}^\infty |a_n|^2 n! <\infty$, under an identification of $p^{\circ n}$ with $p^n$.
\end{prop}
\begin{proof}[Proof]
Clearly, any element in the $n$-th level $\bigH^{\circ n}$ can be written as $p^{\circ n} a$ for some $a \in \mathbb H$, and
\[
\langle p^{\circ n},p^{\circ n} \rangle =
\frac{1}{n!}\sum_{\sigma,\tau \in S_n}\langle \langle\cdots\langle \langle \langle p,p\rangle p,p \rangle p,p \rangle  \cdots \rangle p,p \rangle = n!
\]
\end{proof}

\section{The slice monogenic case}
\label{sec5}
\setcounter{equation}{0}

In this section we recall just the definition and some properties of slice monogenic functions
and  we show how the results obtained in Section 2 can be reformulated in this case.
We work with the real Clifford algebra $\rr_n$ over $n$ imaginary units
$e_1,\dots,e_n$ satisfying the relations $e_ie_j+e_je_i=-2\delta_{ij}$. An element in the
Clifford algebra $\rr_n$ is of the form $\sum_A e_Ax_A$ where
$A=i_1\ldots i_r$, $i_\ell\in \{1,2,\ldots, n\}$, $i_1<\ldots <i_r$ is a multi-index,
$e_A=e_{i_1} e_{i_2}\ldots e_{i_r}$ and $e_\emptyset =1$. We set $|A|=i_1+\ldots +i_r$ and we call $k$-vectors the elements of the form $\sum_{A,\, |A|=k} e_Ax_A$, if $k>0$.
In the Clifford algebra $\mathbb{R}_n$, we can identify some specific elements with the vectors in the Euclidean space $\mathbb{R}^{n+1}$:
an element $(x_0,x_1,\ldots,x_n)\in \rr^{n+1}$  will be identified with the element
 $
 x=x_0+\underline{x}=x_0+ \sum_{j=1}^nx_je_j
 $
called, in short, paravector.
The norm of $x\in\rr^{n+1}$ is defined as $|x|^2=x_0^2+x_1^2+\ldots +x_n^2$. The real part  $x_0$ of $x$ will be also denoted by ${\rm Re}(x)$.
Using the above identification, a function $f:\ U\subseteq \rr^{n+1}\to\rr_n$ is seen as
a function $f(x)$ of the paravector $x$.
We will denote by $\mathbb{S}$ the $(n-1)$-dimensional sphere of unit 1-vectors in $\mathbb{R}^n$, i.e.
$$
\mathbb{S}=\{ e_1x_1+\ldots +e_nx_n\ :\  x_1^2+\ldots +x_n^2=1\}.
$$
Note that to any nonreal paravector $x=x_0+e_1x_1+\ldots +e_nx_n$ we can associate a $(n-1)$-dimensional sphere defined as the set, denoted by $[x]$, of elements of the form
$x_0+I|e_1x_1+\ldots +e_nx_n|$ when $I$ varies in $\mathbb S$.\\
As it is well known, for $n\geq 3$ the Clifford algebra $\mathbb{R}_n$ contains zero divisors. Thus, in general, the result which hold in the quaternionic setting do not necessarily hold in Clifford algebra. For this reasons, we quickly revise the definitions and results given for the quaternionic Fock space. We omit the proofs since, as the reader may easily check, the proofs given in the quaternionic case  are valid also in this setting.\\
We begin by giving the definition of slice monogenic functions (see \cite{bookfunctional}).
\begin{defn}
Let $U\subseteq\rr^{n+1}$ be an open set and let
$f:\ U\to\rr_n$ be a real differentiable function. Let
$I\in\mathbb{S}$ and let $f_I$ be the restriction of $f$ to the
complex plane $\mathbb{C}_I$.
We say that $f$ is a (left) slice monogenic function if for every
$I\in\mathbb{S}$, we have
$$
\frac{1}{2}\left(\frac{\partial }{\partial u}+I\frac{\partial
}{\partial v}\right)f_I(u+Iv)=0,
$$
on $U\cap \mathbb{C}_I$. The set of (left) slice monogenic functions on $U$ will be denoted by $\mathcal{SM}(U)$.
\end{defn}

The slice monogenic Fock spaces and their properties are as follows.

\begin{defn}[Slice hyperholomorphic Clifford-Fock space]
Let $I$ be any element in $\mathbb{S}$. Consider the set
$$
\mathcal{F}(\mathbb{R}^{n+1})=\{f\in \mathcal{SM}(\mathbb{R}^{n+1})\ |\  \int_{\C_I} e^{-|x|^2} |f_I(x)|^2 d\sigma(u,v)<\infty\}
$$
where $x=u+Iv$, $d\sigma(u,v):=\frac{1}{\pi}dudv$.
We will call $\mathcal{F}(\mathbb{R}^{n+1})$ (slice hyperholomorphic) Clifford-Fock space.
\end{defn}
We endow $\mathcal{F}(\mathbb{R}^{n+1})$ with the inner product (which does not depend on the choice of the imaginary unit $I\in\mathbb{S}$):
\begin{equation}\label{innerCLIF}
\langle f, g\rangle:=\int_{\C_I} e^{-|x|^2} \overline{g_I(x)}f_I(x) d\sigma(u,v).
\end{equation}

\begin{prop}
The Clifford-Fock space $\mathcal{F}(\mathbb{R}^{n+1})$ contains the monomials $x^m$, $m\in\mathbb{N}$ which form an orthogonal basis, where $x$ is a paravector in $\mathbb R^{n+1}$.
\end{prop}

Starting from the function $e^{zy}=\sum_{m=0}^{+\infty} \frac{z^my^m}{m!}$, holomorphic in $z$ that can be interpreted as an element on a complex plane $\C_I$
we can extend it to a slice monogenic function as
$$
{\rm ext}(e^{zy})=\frac 12 ( 1- I_xI ) \sum_{m=0}^{+\infty} \frac{z^my^m}{m!}+ \frac 12 ( 1+ I_xI )  \sum_{m=0}^{+\infty}\frac{{\bar z}^my^m}{m!}=e_\star^{xy}
$$
and since $y$ is arbitrary, we get the  function we need.
We now consider the function $k_y(x):=e_\star^{x\bar y}$ and we have the reproducing property in the Clifford-Fock space.
\begin{thm}
For every $f\in\mathcal{F}(\mathbb{R}^{n+1})$ we have
$$
\langle f, k_x \rangle= f(x).
$$
Moreover, $\langle k_x,k_y\rangle= e_\star^{y\bar{x}}$.
\end{thm}

\begin{prop}
A function $f(x)=\sum_{m=0}^{+\infty} x^m a_m$, $a_m\in \mathbb{R}_n$ for $m\in \mathbb{N}$, belongs to $\mathcal{F}(\mathbb{R}^{n+1})$ if and only if $\sum_{m=0}^{+\infty}  |a_m|^2 m!<\infty$.
\end{prop}
In the case of modules over $\mathbb R_n$ the full Fock module is still under investigation.


\begin{thebibliography}{99}

\bibitem{adler}
 S. Adler, {\em Quaternionic Quantum Field Theory}, Oxford University Press,
1995.

\bibitem{MR2002b:47144}
D. Alpay,
{\em The {S}chur algorithm, reproducing kernel spaces and system
  theory},
 American Mathematical Society, Providence, RI, 2001.
 Translated from the 1998 French original by Stephen S. Wilson,
  Panoramas et Synth\`eses.

\bibitem{abcsNP} D. Alpay, V. Bolotnikov, F. Colombo, I. Sabadini,
 {\em Self-mappings of the
quaternionic unit ball: multiplier properties, Schwarz-Pick inequality,  and
Nevanlinna--Pick interpolation problem}, to appear in Indiana Univ. Math. J., arXiv:1308.2658.


\bibitem{acls} D. Alpay, F. Colombo, I. Lewkowicz, I. Sabadini,
 {\em Realizations of slice hyperholomorphic generalized contractive and positive functions},
 preprint 2013.  Submitted.


\bibitem{acsxx}
D. {Alpay}, F. {Colombo},  I. {Sabadini},
{\em Inner product spaces and Krein spaces
in the quaternionic setting},
in {\em Recent advances in inverse scattering, Schur analysis and stochastic processes. A collection of papers dedicated to Lev Sakhnovich}, Operator Theory Advances and Applications. Linear Operators and Linear Systems, 2014.


\bibitem{acs1}
D. {Alpay}, F. {Colombo},  I. {Sabadini},
\newblock {\em Schur functions and their realizations in the slice hyperholomorphic
  setting},
\newblock {Integral Equations Operator Theory}, {\bf 72}  (2012), 253--289.

\bibitem{acs2}
D. {Alpay}, F. {Colombo},  I. {Sabadini},
\newblock {\em Pontryagin  De Branges Rovnyak spaces of slice hyperholomorphic
functions}, Journal d'Analyse Math\'ematique
\newblock{vol. 121 (2013), no. 1, 87-125.}

\bibitem{acs3}
D. {Alpay}, F. {Colombo}, I. {Sabadini},
\newblock {\em Krein-Langer factorization and related topics in the slice
  hyperholomorphic setting},
 J. Geom. Anal., {\bf 24} (2014), no. 2, 843--872.

\bibitem{adrs}
D. Alpay, A. Dijksma, J. Rovnyak, and H. de Snoo,
\newblock {\em {Schur} functions, operator colligations, and reproducing kernel
  {P}ontryagin spaces}, volume~96 of {\em Operator theory: {A}dvances and
  {A}pplications}.
\newblock Birkh{\" a}user Verlag, Basel, 1997.

\bibitem{2013arXiv1311.3239A}
D.~{Alpay}, P.~{Jorgensen}, and G.~{Salomon}.
\newblock {On free stochastic processes and their derivatives}.
\newblock {\em ArXiv e-prints}, November 2013.

\bibitem{ak} D. Alpay, H. T. Kaptanoglu,
{\em Quaternionic Hilbert spaces and a von Neumann inequality},
Complex Var. Elliptic Equ., {\bf 57} (2012), 667--675.

\bibitem{MR3038506}
D.~Alpay and G.~Salomon.
\newblock Non-commutative stochastic distributions and applications to linear
  systems theory.
\newblock {\em Stochastic Process. Appl.}, 123(6):2303--2322, 2013.

\bibitem{MR0157250}
V.~Bargmann.
\newblock On a {H}ilbert space of analytic functions and an associated integral
  transform.
\newblock {\em Comm. Pure Appl. Math.}, 14:187--214, 1961.

\bibitem{bargmann}
V.~Bargmann.
\newblock Remarks on a {H}ilbert space of analytic functions.
\newblock {\em Proceedings of the {N}ational {A}cademy of {Arts}}, 48:199--204,
  1962.

\bibitem{MR2540072}
M. Bo{\.z}ejko and E. Lytvynov.
\newblock Meixner class of non-commutative generalized stochastic processes
  with freely independent values. {I}. {A} characterization.
\newblock {\em Comm. Math. Phys.}, 292(1):99--129, 2009.

\bibitem{MR2770019}
M. Bo{\.z}ejko and E. Lytvynov.
\newblock Meixner class of non-commutative generalized stochastic processes
  with freely independent values {II}. {T}he generating function.
\newblock {\em Comm. Math. Phys.}, 302(2):425--451, 2011.

\bibitem{Bourbaki}
N.~Bourbaki.
\newblock {\em Algebra {I}. {C}hapters 1--3}.
\newblock Translated from the French. Reprint of the 1989 English translation.
\newblock Elements of Mathematics (Berlin). Springer-Verlag, Berlin, 1998.

\bibitem{CGCS1}
F. Colombo, J. O. Gonzalez-Cervantes, I. Sabadini,
{\em  On slice biregular functions and isomorphisms of Bergman spaces}, Complex Variables and Elliptic Equations, {\bf 57}   (2012), 825--839.

\bibitem{CGCS2}
F. Colombo, J. O. Gonzalez-Cervantes, I. Sabadini,
{\em The C-property for slice regular functions and applications to the Bergman space}, Complex Variables and Elliptic Equations, {\bf 58} (2013), 1355--1372.

\bibitem{INDAM}
F. Colombo, J. O. Gonzalez-Cervantes, M.E. Luna-Elizarraras,  I. Sabadini, M.V. Shapiro,
{\em On two approaches to the Bergman theory for slice regular functions},
Advances in Hypercomplex Analysis; Springer INdAM Series 1, (2013), 39--54.


\bibitem{bookfunctional} F. Colombo, I. Sabadini, D.C. Struppa,
 {\em Noncommutative Functional Calculus. Theory and Applications of Slice Hyperholomorphic Functions},
 Progress in Mathematics, Vol. 289, Birkh\"auser,   2011.


\bibitem{DSS}
R. Delanghe, F.  Sommen, V.  Soucek,
{\em  Clifford algebra and spinor-valued functions. A function theory for the Dirac operator},
  Mathematics and its Applications, 53. Kluwer Academic Publishers Group, Dordrecht, 1992.


\bibitem{ds}
N. Dunford, J. Schwartz, {\em Linear operators, part I: general
theory }, J. Wiley and Sons (1988).


\bibitem{Fock} V. Fock.
\newblock {Verallgemeinerung und L\"osung der Diracschen statistischen Gleichung},
\newblock {\em Zeitschrift f\"ur Physik}.
49: 339--357, 1928.


\bibitem{MR1746976}
F. Hiai and D. Petz.
\newblock {\em The semicircle law, free random variables and entropy},
  volume~77 of {\em Mathematical Surveys and Monographs}.
\newblock American Mathematical Society, Providence, RI, 2000.


\bibitem{MR1244577}
T.~Hida, H.~Kuo, J.~Potthoff, and L.~Streit.
\newblock {\em White noise}, volume 253 of {\em Mathematics and its
  Applications}.
\newblock Kluwer Academic Publishers Group, Dordrecht, 1993.
\newblock An infinite-dimensional calculus.



\bibitem{MR2444857}
T.~Hida and Si~Si.
\newblock {\em Lectures on white noise functionals}.
\newblock World Scientific Publishing Co. Pte. Ltd., Hackensack, NJ, 2008.

\bibitem{HiHi93}
T. Hida and M. Hitsuda.
\newblock {\em Gaussian processes}, volume 120 of {\em Translations of
  Mathematical Monographs}.
\newblock American Mathematical Society, Providence, RI, 1993.
\newblock Translated from the 1976 Japanese original by the authors.

\bibitem{MR1408433}
H.~Holden, B.~{\O}ksendal, J.~Ub{\o}e, and T.~Zhang.
\newblock {\em Stochastic partial differential equations}.
\newblock Probability and its Applications. Birkh\"auser Boston Inc., Boston,
  MA, 1996.


\bibitem{MR768240}
L.~P. Horwitz and L.~C. Biedenharn.
\newblock Quaternion quantum mechanics: second quantization and gauge fields.
\newblock {\em Ann. Physics}, 157(2):432--488, 1984.

\bibitem{MR1352898}
L.~P. Horwitz and A.~Razon.
\newblock Tensor product of quaternion {H}ilbert modules.
\newblock In {\em Classical and quantum systems (Goslar, 1991)}, pages
  266--268. World Sci. Publishing, River Edge, NJ, 1993.

\bibitem{Jacobson}
N.~Jacobson.
\newblock {\em Basic algebra. {II}}.
\newblock Second edition. W. H. Freeman and Company, New York, 1989.


\bibitem{MR0076317}
I.~E. Segal.
\newblock Tensor algebras over {H}ilbert spaces. {I}.
\newblock {\em Trans. Amer. Math. Soc.}, 81:106--134, 1956.

\bibitem{MR0077908}
I.~E. Segal.
\newblock Tensor algebras over {H}ilbert spaces. {II}.
\newblock {\em Ann. of Math. (2)}, 63:160--175, 1956.


\bibitem{MR2227827}
D. Voiculescu.
\newblock Aspects of free probability.
\newblock In {\em X{IV}th {I}nternational {C}ongress on {M}athematical
  {P}hysics}, pages 145--157. World Sci. Publ., Hackensack, NJ, 2005.


\bibitem{MR1217253}
D.~V. Voiculescu, K.~J. Dykema, and A.~Nica.
\newblock {\em Free random variables}, volume~1 of {\em CRM Monograph Series}.
\newblock American Mathematical Society, Providence, RI, 1992.
\newblock A noncommutative probability approach to free products with
  applications to random matrices, operator algebras and harmonic analysis on
  free groups.

\bibitem{MR1851117}
Zhi-yuan Huang and Jia-an Yan.
\newblock {\em Introduction to infinite dimensional stochastic analysis},
  volume 502 of {\em Mathematics and its Applications}.
\newblock Kluwer Academic Publishers, Dordrecht, chinese edition, 2000.


\end{thebibliography}

\end{document}